\documentclass[smallextended,envcountsect]{svjour3}
\smartqed 
\usepackage{graphicx}
\usepackage{amsmath,amssymb,mathtools}
\usepackage{hyperref}
\usepackage{xcolor}
\usepackage{setspace}
\usepackage{cite}
\journalname{JOTA}
\colorlet{mylinkcolor}{red!80!black}
\colorlet{myurlcolor}{green!50!black}
\colorlet{mysectioncolor}{blue!50!black}
\hypersetup{
    colorlinks=true,
    linkcolor=mylinkcolor,
    urlcolor=myurlcolor,
    hypertexnames,
}

\newcommand{\HOL}{H{\"o}lder}
\newcommand{\mynor}{\lVert\cdot\rVert}
\newcommand{\RRb}{\mathbb{R}}
\newcommand{\xxb}{\bar{x}}
\newcommand{\yyb}{\bar{y}}
\newcommand{\MFNU}{M_f(\nu)}
\newcommand{\LFNU}{L_f(\nu)}

\newcommand{\fracnu}{\left(\frac{1+\nu}{2\nu}\right)^\nu}

\newcommand{\vvh}{\hat{v}}
\newcommand{\BBt}{\tilde{B}}
\newcommand{\EEc}{\mathcal{E}}
\newcommand{\BBb}{\mathbb{B}}
\newcommand{\SSb}{\mathbb{S}}
\newcommand{\xopt}{\bar{x}}
\newcommand{\xh}{\hat{x}}

\begin{document}

\title{On the Quality of First-Order Approximation of Functions with H{\"o}lder Continuous Gradient}

\author{%
Guillaume~O.~Berger\and
P.-A.~Absil\and
Rapha\"el~M.~Jungers\and
Yurii~Nesterov}

\institute{%
Guillaume~O.~Berger, Corresponding author \and P.-A.~Absil \and Rapha\"el~M.~Jungers \and Yurii~Nesterov \at
UCLouvain, Louvain-la-Neuve, Belgium \\
\{guillaume.berger,pa.absil,raphael.jungers,yurii.nesterov\}@uclouvain.be}

\date{Received: date / Accepted: date}
%The correct dates will be entered by the editor.

\maketitle

%\tempcomm{Hello I am a supplementary material and in the pdf all references are present}

\begin{abstract}
We show that \HOL{} continuity of the gradient is not only a sufficient condition, but also a necessary condition for the existence of a global upper bound on the error of the first-order Taylor approximation.
We also relate this global upper bound to the \HOL{} constant of the gradient.
This relation is expressed as an interval, depending on the \HOL{} constant, in which the error of the first-order Taylor approximation is guaranteed to be.
We show that, for the Lipschitz continuous case, the interval cannot be reduced.
An application to the norms of quadratic forms is proposed, which allows us to derive a novel characterization of Euclidean norms.
\end{abstract}



%All acknowledgements should be placed in the back of the paper after Conclusions..

%%%%%%%%%%%%%%%%%%%%%%%%%%%%%%%%%%%%%%%%%%%%%%%%%%%%%%%%%%%%%%%%%%%%%%%%%%%%%%%%%%%%%%%%%%%%%%%%%%%%
\section{Introduction}\label{sec-introduction}

The purpose of this paper is to investigate the relation between two properties of a real-valued function, which play crucial roles in optimization.
The first property of interest is the \HOL{} continuity of the gradient, which means that the variation of the gradient of the function between two points is upper bounded by a power (with exponent smaller than or equal to one) of the distance between the two points; up to some multiplicative constant, called the \HOL{} constant of the gradient.
The second property is that there exists a global upper bound on the error of the first-order Taylor approximation of the function.
This global upper bound takes the form of a power (with exponent between one and two) of the distance between the point of interest and the reference point for the Taylor approximation; the power of the distance can be scaled by a multiplicative constant, called the approximation parameter of the function.

The class of functions with \HOL{} continuous gradient is ubiquitous in optimization.
Indeed, the vast majority of first-order optimization methods (e.g., the gradient descent) requires \HOL{} continuity of the gradient to compute the optimal step size and to assert the convergence of the method to a stationary point \cite{nesterov2013gradient,yashtini2016global,cartis2017worst,nesterov2015universal}.
It is well known that, if a function has \HOL{} continuous gradient, then there is a global upper bound on the error of the first-order Taylor approximation of the function at any point.
Actually, it appears that the majority of the above-mentioned developments only make use of the \HOL{} continuity of the gradient as a convenient sufficient condition to ensure the global upper bound on the error of the first-order Taylor approximation.

For example, in the global convergence analysis found in \cite[Section~3]{yashtini2016global}, the function is assumed to have \HOL{} continuous gradient, and the very first step is a lemma, stating that this property implies the existence of a global upper bound on the error of the first-order Taylor approximation of the function.
In fact, the main complexity result~\cite[Corollary~2]{yashtini2016global} can be obtained by assuming the global upper bound on the error of the first-order Taylor approximation, while disregarding the \HOL{} continuity of the gradient.

Another evidence of the prominent importance of the second property (global upper bound on the error of the first-order Taylor approximation) over the first one (\HOL{} continuity of the gradient) is that it is a generalization of the second property---and not of the first one---that is used as an assumption in \cite{boumal2018global} to generalize global complexity bounds for the minimization of functions defined on Riemannian manifolds.

As already mentioned, the first property is sufficient for the second one.
The following questions then naturally arise.
Is it also necessary?
And, if it is, how can we relate the \HOL{} constant of the gradient to the approximation parameter arising in the second property?
To the best of the authors' knowledge, an answer to this second question is available in the literature only for convex functions with \HOL{} exponent of the gradient equal to one (the particular case of \HOL{} continuity with exponent equal to one is generally referred to as Lipschitz continuity); see, e.g., \cite[Theorem~2.1.5]{nesterov2013introductory}.
However, this second question is important for its implications in optimization: The upper bound on the error of the first-order Taylor approximation is used to compute the step size and to estimate the global rate of convergence of first-order optimization methods; the more accurately we know this bound, the better we can choose the step size and estimate the rate of convergence (see, e.g., Example~\ref{exa-gradient-method}).

In this paper, we provide an answer to the above questions.
We show in Theorem~\ref{thm-holder-general} that a function satisfies the second property, if \emph{and only if} it satisfies the first one.
We also provide an interval, depending on the \HOL{} constant of the gradient, in which the approximation parameter of the function is guaranteed to be.
We show that, in the case of functions with Lipschitz continuous gradient and quadratic upper bound on the first-order Taylor approximation, this interval is tight (see Example~\ref{exa-infinite}).
We also provide a more detailed analysis, when the domain of the function is endowed with a Euclidean norm, i.e., a norm induced by a scalar product.
Finally, we apply these results to quadratic functions; see Section~\ref{sec-app-bilin}.
This allows us to obtain a novel characterization of Euclidean norms.

The paper is organized as follows.
In Section~\ref{sec-definitions}, we introduce notation and definitions.
The questions we address are motivated in Section~\ref{sec-motivation} with an example that demonstrates their importance in optimization.
In Section~\ref{sec-theorem}, we prove the equivalence between the two properties described above, and explain the relation between the \HOL{} constant of the gradient and the approximation parameter.
In Section~\ref{sec-corollaries}, we particularize the results of Section~\ref{sec-theorem} to the case of functions with Lipschitz continuous gradient, and we show that the bounds, derived in this specific case, are tight.
The application to quadratic functions is presented in Section~\ref{sec-app-bilin}.

%%%%%%%%%%%%%%%%%%%%%%%%%%%%%%%%%%%%%%%%%%%%%%%%%%%%%%%%%%%%%%%%%%%%%%%%%%%%%%%%%%%%%%%%%%%%%%%%%%%%
\section{Notation and Preliminaries}\label{sec-definitions}

In the sequel, $E$ is a real finite-dimensional normed vector space with norm $\mynor$.
The norm $\mynor$ is said to be \emph{Euclidean}, if it is induced by a scalar product.
Equivalently, $\mynor$ is Euclidean, if and only if every pair of vectors $u,v\in E$ satisfies the parallelogram law $\lVert u+v\rVert^2 + \lVert u-v\rVert^2 = 2\lVert u\rVert^2 + 2\lVert v\rVert^2$ (Jordan--von Neumann theorem \cite{jordan1935inner}).
In this paper, we will consider both cases of Euclidean and non-Euclidean norms.

On $\RRb$, we denote the absolute value by $\lvert\cdot\rvert$.
The {\itshape dual} of $E$ is the space of linear maps from $E$ to $\RRb$ and is denoted by $E^*$.
We denote by $\langle \varphi,h\rangle$ the image of $h\in E$ by $\varphi\in E^*$.
We endow $E^*$ with the \emph{dual norm}, $\mynor_*$, defined by
\[
\lVert \varphi \rVert_* = \max\, \{ \langle \varphi,h\rangle : h\in E,\,\lVert h\rVert = 1 \} .
\]

Let $f$ be a function from $E$ to $\RRb$.
In the sequel, we will always assume that $f$ is differentiable.
For $x\in E$, we denote by $f'(x)$ the derivative---which we also term \emph{gradient} as in~\cite{nesterov2013gradient}---of $f$ at $x$.
Note that $f'(x)\in E^*$.
(Thus it acts on vectors $h\in E$, and $\langle f'(x),h\rangle$ is sometimes called the directional derivative of $f$ at $x$, in the direction $h$.)

The first property of interest is defined by:

\begin{definition}[$\nu$-\HOL{} continuous gradient]\label{def-holder}
Let $0\leq\nu\leq1$.
We say that $f$ has \emph{$\nu$-\HOL{} continuous gradient}, if there exists an $M\geq0$, such that, for every $x,y\in E$,
\begin{equation}\label{eq-holder-def}
\lVert f'(x) - f'(y) \rVert_* \leq M \lVert x-y \rVert^\nu .
\end{equation}
In this case, it is easily seen that there exists a smallest $M$ satisfying inequality~\eqref{eq-holder-def} for every $x,y\in E$.
This $M$ is called the \emph{$\nu$-\HOL{} constant} of the gradient of $f$, and is denoted by $\MFNU$:
\[
\MFNU \coloneqq \sup_{x\neq y} \frac{\lVert f'(x) - f'(y) \rVert_*}{\lVert x-y \rVert^\nu} < \infty .
\]\vskip0pt
\end{definition}

Note that not every function has $\nu$-\HOL{} continuous gradient for some $\nu\in[0,1]$.%
\footnote{E.g., $x\in\RRb\mapsto x^3$, or $x\in\RRb\mapsto x^2\sin(1/x^2)$ (with continuous extension at $0$).}
The next definition particularizes the former for $\nu=1$.

\begin{definition}[Lipschitz continuous gradient]\label{def-lipsch}
We say that $f$ has \emph{Lip\-schitz-continuous gradient}, if $f$ has $1$-\HOL{} continuous gradient.
In this case, $M_f(1)$ is called the \emph{Lipschitz constant} of the gradient of $f$.
\end{definition}

As for the second property of interest in this paper, it reads as:

\begin{definition}[$\nu$-approximable]\label{def-approx-parameter}
We will say that $f$ is \emph{$\nu$-approximable}, if there exists an $L\geq0$, such that, for every $x,y\in E$,
\begin{equation}\label{eq-error-approx}
\lvert f(y) - f(x) - \langle f'(x),y-x\rangle \rvert \leq \frac{L}{1+\nu} \lVert y-x \rVert^{1+\nu} .
\end{equation}
The smallest such $L$ will be called the \emph{$\nu$-approximation parameter} of $f$, and will be denoted by $\LFNU$:
\[
\LFNU \coloneqq (1+\nu)\: \sup_{x\neq y} \frac{\lvert f(y) - f(x) - \langle f'(x),y-x \rangle \rvert}{\lVert y-x \rVert^{1+\nu}} < \infty .
\]\vskip0pt
\end{definition}

\begin{remark}\label{rem-affine-invariance}
Clearly, Definitions~\ref{def-holder}--\ref{def-approx-parameter} are invariant under translation, and addition of linear functions.
In other words, if $g(x)={f(x+a)}+\langle \varphi,x\rangle + c$ for some $a\in E$, $\varphi\in E^*$ (i.e., $\varphi$ is a linear map from $E$ to $\RRb$), and $c\in\RRb$, then $g$ is $\nu$-approximable (resp.~has $\nu$-\HOL{} continuous gradient), if and only if $f$ is $\nu$-approximable (resp.~has $\nu$-\HOL{} continuous gradient); and moreover, $M_g(\nu)=\MFNU$, and $L_g(\nu)=\LFNU$.
\end{remark}

The following proposition is a classical result in optimization.
It states that, if $f$ has $\nu$-\HOL{} continuous gradient for some $0\leq\nu\leq1$, then $f$ is $\nu$-approximable.
Moreover, the $\nu$-\HOL{} constant is an upper bound on the $\nu$-approximation parameter:

\begin{proposition}\label{prop-bound-integration}
Let $f$ have $\nu$-\HOL{} continuous gradient.
Then, $f$ is $\nu$-approximable, and $\LFNU \leq \MFNU$.
\end{proposition}

\begin{proof}
See \cite[Lemma~1]{yashtini2016global}, for instance.~\qed
\end{proof}

%%%%%%%%%%%%%%%%%%%%%%%%%%%%%%%%%%%%%%%%%%%%%%%%%%%%%%%%%%%%%%%%%%%%%%%%%%%%%%%%%%%%%%%%%%%%%%%%%%%%
\section{Motivation}\label{sec-motivation}

We present an example of optimization method, which uses inequality~\eqref{eq-error-approx} to assert the convergence of the method to a stationary point.
To compute the step size, the method requires an $L$ satisfying \eqref{eq-error-approx}.
For a given such $L$, the global complexity of the method is proportional to $L^{1/\nu}$.
It is thus \emph{beneficial to choose $L$ as small as possible}, ideally $L=\LFNU$.
The results in the next section provide an interval (or a specific value), depending on $\MFNU$, in which $\LFNU$ is guaranteed to be.

\begin{example}\label{exa-gradient-method}
(from \cite{yashtini2016global})
Let $f$ be a real-valued function, defined on $E=\RRb^n$ (with any norm $\mynor$), and satisfying \eqref{eq-error-approx} for some $L\geq0$ and some%
\footnote{Note that the method requires $\nu>0$.
In fact, finding a descent direction for a non-smooth non-convex function is NP-hard \cite{nesterov2013gradient}, and thus, it is reasonable to ask that $\nu>0$.}
$0<\nu\leq1$ (i.e., $f$ is $\nu$-approximable with parameter $\LFNU\leq L$).
Suppose that there is a lower bound $f_*\in\RRb$ on $f$: $f(x)\geq f_*$ for all $x\in E$.
We use the gradient method to find a stationary point of $f$.
The step size depends on $L$ and $\nu$, and on the norm of the gradient at each iteration point.

The method goes as follows.
(For more details, we refer the reader to \cite{yashtini2016global}.)
We start from a point $x_0\in\RRb^n$.
For $k=0,1,2,\dots$, let $x_k$ be the iterate at step $k$.
The norm of $f'(x_k)$ is denoted by $n_k = \lVert f'(x_k)\rVert_*$.
Let $d_k\in\RRb^n$ be such that $\lVert d_k\rVert=1$ and $\langle f'(x_k),d_k\rangle=n_k$ (i.e., $-d_k$ is a steepest descent direction).

Fix $0<\xi<1$.
Define the step size at iteration $k$ as $h_k=\xi\left(\frac{1+\nu}{L}\right)^{1/\nu}n_k^{1/\nu}$.
Then, the next iterate is defined by $x_{k+1} = x_k - h_kd_k$.
From \eqref{eq-error-approx}, we have
\[
f(x_{k+1}) \leq f(x_k) - \langle f'(x_k),h_kd_k\rangle + \frac{L}{1+\nu} \lVert h_kd_k\rVert^{1+\nu} .
\]
This gives
\begin{align*}
f(x_{k+1}) &\leq f(x_k) - \xi\left(\frac{1+\nu}{L}\right)^{1/\nu}n_k^{1+1/\nu} + \xi^{1+\nu} \left(\frac{1+\nu}{L}\right)^{1/\nu} n_k^{1+1/\nu} \\
&= f(x_k) - \xi (1-\xi^\nu) \left(\frac{1+\nu}{L}\right)^{1/\nu} n_k^{1+1/\nu} .
\end{align*}
Summing over $k$ from $0$ to $K$, we get
\[
\xi (1-\xi^\nu) \left(\frac{1+\nu}{L}\right)^{1/\nu} \sum_{k=0}^K n_k^{1+1/\nu} \leq f(x_0) - f(x_{K+1}) \leq f(x_0) - f_*.
\]
We conclude that
\[
\min_{0\leq k\leq K} \lVert f'(x_k)\rVert_*^{1+1/\nu} \leq \frac{1}{K+1} \left(\frac{L}{1+\nu}\right)^{1/\nu} \frac{f(x_0) - f_*}{\xi (1-\xi^\nu)} .
\]
In particular, if we choose $\xi=\left(\frac{1}{1+\nu}\right)^{1/\nu}$, we obtain
\[
\min_{0\leq k\leq K} \lVert f'(x_k)\rVert_*^{1+1/\nu} \leq \frac{1}{K+1} \frac{1+\nu}{\nu} L^{1/\nu} (f(x_0) - f_*) .
\]

The objective is to converge to a quasi-stationary point: For some fixed $\epsilon>0$, we want to find an $\xxb$ such that $\lVert f'(\xxb) \rVert_*\leq \epsilon$.
Then, the above-presented method stops after a number of iterations not greater than
\begin{equation}\label{eq-iter}
\left\lceil \frac{1}{\epsilon^{1+1/\nu}} \frac{1+\nu}{\nu} L^{1/\nu} (f(x_0) - f_*) \right\rceil,
\end{equation}
where $\lceil\cdot\rceil$ denotes the ceiling operator.
\end{example}

In summary, with the above example, we have shown that:
\begin{itemize}
    \item The existence of an $L$ satisfying \eqref{eq-error-approx} (i.e., being $\nu$-approximable) is sufficient to assert the convergence of the method to a quasi-stationary point.
    Contrary to the way it is presented in many textbooks, it is not necessary to resort to the $\nu$-\HOL{} continuous gradient assumption~\eqref{eq-holder-def}.
    \item The bound~\eqref{eq-iter} on the number of iterations shows that the knowledge of the smallest $L$ satisfying \eqref{eq-error-approx} (i.e., $\LFNU$) allows us to have a better bound on the total number of iterations required by the method.
\end{itemize}

As we will see in the next section, a lower bound on $\LFNU$ can be obtained from the $\nu$-\HOL{} constant of the gradient.
In the above developments, the derivation of \eqref{eq-iter} requires only a global upper bound on the difference between the function and its first-order approximation, and not on the absolute value thereof, as in \eqref{eq-error-approx}.
A more general class of functions can be analyzed, if we allow for different parameters in the lower bound and in the upper bound on the difference between the function and its first-order approximation.
For this class of functions, we obtain similar conclusions as for the case of $\nu$-approximable functions (see next section as well).

%%%%%%%%%%%%%%%%%%%%%%%%%%%%%%%%%%%%%%%%%%%%%%%%%%%%%%%%%%%%%%%%%%%%%%%%%%%%%%%%%%%%%%%%%%%%%%%%%%%%
\section{Main Results}\label{sec-theorem}

The main contributions of this paper are summarized in Theorems~\ref{thm-holder-general}--\ref{thm-holder-real}.
Theorem~\ref{thm-holder-general} shows that $\nu$-approximability implies $\nu$-\HOL{} continuity of the gradient, and also provides an upper bound on the \HOL{} constant of the gradient.
From this upper bound, we obtain an interval, depending on $\MFNU$, in which the value of $\LFNU$ is guaranteed to be.
We also show that a smaller interval can be obtained, if we further assume that the norm $\mynor$, on the domain of $f$, is Euclidean.

\begin{theorem}\label{thm-holder-general}
Let $0\leq\nu\leq1$.
Let $f$ be a differentiable function from $E$ (with norm $\mynor$) to $\RRb$.
Suppose that $f$ is $\nu$-approximable, and denote by $\LFNU$ its $\nu$-approximation parameter.
Then, $f$ has $\nu$-\HOL{} continuous gradient with $\nu$-\HOL{} constant $\MFNU$ satisfying%
\footnote{With the convention that $0^0=1$, hence, $\left(\frac{1+\nu}{\nu}\right)^\nu=\left(\frac{1+\nu}{\nu}\right)^{\nu/2}=1$, when $\nu=0$.}
\begin{equation}\label{eq-coeff-1}
\LFNU \leq \MFNU \leq 2^{1-\nu} \left(\frac{1+\nu}{\nu}\right)^\nu \LFNU.
\end{equation}
Moreover, if we further assume that $\mynor$ is Euclidean, then
\begin{equation}\label{eq-coeff-2}
\LFNU \leq \MFNU \leq \frac{2^{1-\nu}}{\sqrt{1+\nu}} \left(\frac{1+\nu}{\nu}\right)^{\nu/2} \LFNU .
\end{equation}\vskip0pt
\end{theorem}

See Figure~\ref{fig-fnu} for a comparison of \eqref{eq-coeff-1} and \eqref{eq-coeff-2}.

\begin{proof}
Let us fix arbitrary $\xxb,\yyb\in E$.
From Remark~\ref{rem-affine-invariance}, we may assume, without loss of generality, that $\xxb = 0$, $f(\xxb) = 0$, and $f'(\xxb) = 0$.
Let $\varphi = f'(\yyb)$.
Then, for any two directions $z_1,z_2\in E$, we have (from \eqref{eq-error-approx} with $(x,y)=(0,z_1)$, and with $(x,y)=(\yyb,z_1)$)
\[
\frac{-\LFNU}{1+\nu} \lVert z_1 \rVert^{1+\nu} \leq f(z_1) \leq f(\yyb) + \langle \varphi,z_1-\yyb\rangle + \frac{\LFNU}{1+\nu} \lVert z_1-\yyb \rVert^{1+\nu},
\]
and (from \eqref{eq-error-approx} with $(x,y)=(0,z_2)$ and with $(x,y)=(\yyb,z_2)$)
\[
\frac{\LFNU}{1+\nu} \lVert z_2 \rVert^{1+\nu} \geq f(z_2) \geq f(\yyb) + \langle \varphi,z_2-\yyb\rangle - \frac{\LFNU}{1+\nu} \lVert z_2-\yyb \rVert^{1+\nu} .
\]
Subtracting the rightmost and leftmost sides of the first set of inequalities from the second one, we get
\[
\langle \varphi,z_2-z_1\rangle \leq \frac{\LFNU}{1+\nu} \left( \lVert z_1 \rVert^{1+\nu} + \lVert z_2 \rVert^{1+\nu} + \lVert z_1-\yyb \rVert^{1+\nu} + \lVert z_2-\yyb \rVert^{1+\nu} \right) \!.
\]

Let $\vvh\in E$ be such that $\lVert \vvh\rVert = 1$ and $\langle \varphi,\vvh\rangle=\lVert \varphi\rVert_*$.
(From the compactness of $\{x\in E:\lVert x\rVert=1\}$, it is always possible to find such a $\vvh$.)
Then, let $\alpha\geq0$, and define $z_1 = \frac{\yyb-\alpha\vvh}{2}$ and $z_2 = \frac{\yyb+\alpha\vvh}{2}$.
This gives
\begin{equation}\label{eq-bound-1}
\alpha\, \langle \varphi,\vvh\rangle \leq 2\:\frac{\LFNU}{1+\nu} \left( \left\lVert \frac{\yyb-\alpha\vvh}{2} \right\rVert^{1+\nu} + \left\lVert \frac{\yyb+\alpha\vvh}{2} \right\rVert^{1+\nu} \right) \!.
\end{equation}

{\itshape Non-Euclidean case} Let $r=\lVert \yyb\rVert$.
The right-hand side of \eqref{eq-bound-1} can be bounded as follows:
\[
\left\lVert \frac{\yyb-\alpha\vvh}{2} \right\rVert^{1+\nu} + \left\lVert \frac{\yyb+\alpha\vvh}{2} \right\rVert^{1+\nu} \leq 2 \left(\frac{r+\alpha}{2}\right)^{1+\nu} .
\]
If $\nu>0$, then, from \eqref{eq-bound-1} and the above inequality, we obtain (letting $\alpha = \frac{r}{\nu}$)
\[
\langle \varphi,\vvh\rangle \leq 4\:\frac{\LFNU}{1+\nu} \:\frac{\nu}{r} \left(\frac{r(1+\nu)}{2\nu}\right)^{1+\nu} = 2 \LFNU \fracnu r^\nu .
\]
Hence, we have the conclusion in the case of $\nu>0$.
On the other hand, if $\nu=0$, we get $M_f(0)\leq2L_f(0)$, by letting $\alpha\to+\infty$ in \eqref{eq-bound-1}.

{\itshape Euclidean case} Now, assume that $\mynor$ is Euclidean.
Then, using the inequality $a^p + b^p \leq 2^{1-p} \left(a+b\right)^p$ for every $a,b \geq 0$ and $p \in \left[0,1\right]$, (resulting from the concavity of $x \mapsto x^p$), we get the following upper bound on the right-hand side of \eqref{eq-bound-1}:
\[
\left\lVert \frac{\yyb-\alpha\vvh}{2} \right\rVert^{1+\nu} + \left\lVert \frac{\yyb+\alpha\vvh}{2} \right\rVert^{1+\nu} \leq 2^{(1-\nu)/2} \left( \left\lVert \frac{\yyb-\alpha\vvh}{2} \right\rVert^2 + \left\lVert \frac{\yyb+\alpha\vvh}{2} \right\rVert^2 \right)^{(1+\nu)/2} .
\]
Let $r=\lVert \yyb\rVert$.
Using the parallelogram identity, we obtain
\[
\left\lVert \frac{\yyb-\alpha\vvh}{2} \right\rVert^{1+\nu} + \left\lVert \frac{\yyb+\alpha\vvh}{2} \right\rVert^{1+\nu} \leq 2^{(1-\nu)/2} \left( \frac{r^2+\alpha^2}{2} \right)^{(1+\nu)/2} .
\]
For $\nu>0$, we get (letting $\alpha=\frac{r}{\sqrt{\nu}}$)
\[
\langle \varphi,\vvh\rangle \leq \LFNU \frac{2^{1-\nu}}{\sqrt{1+\nu}} \left(\frac{1+\nu}{\nu}\right)^{\nu/2} r^\nu .
\]
This concludes the proof of the theorem.~\qed
\end{proof}

\begin{figure}[ht]
\centering
\includegraphics[scale=0.8]{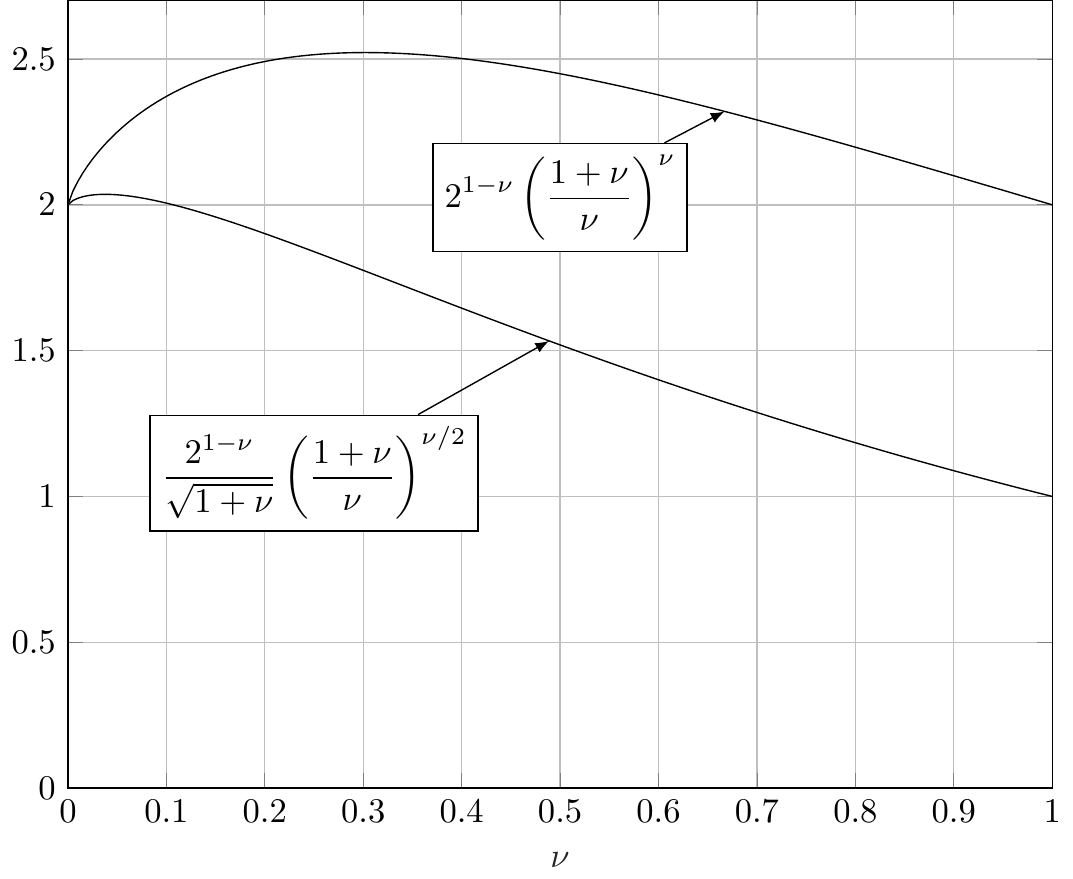}
\caption{Comparison of the coefficients appearing in the upper bounds on the $\nu$-\HOL{} constant in \eqref{eq-coeff-1} and \eqref{eq-coeff-2} in Theorem~\ref{thm-holder-general}.}
\label{fig-fnu}
\end{figure}

We now move to the second main result.
Observe that inequality~\eqref{eq-error-approx} is equivalent to
\begin{equation}\label{eq-two-side-Lipschitz}
\frac{-L}{1+\nu} \lVert y-x \rVert^{1+\nu} \leq f(y) - f(x) - \langle f'(x),y-x\rangle \leq \frac{L}{1+\nu} \lVert y-x \rVert^{1+\nu} .
\end{equation}
In some situations, we may also want to have different values of $L$ for the lower bound and the upper bound in \eqref{eq-two-side-Lipschitz}.
This arises, for example, if we consider convex functions.
In this case, the lower bound is zero, since the first-order approximation of a convex function always lies below the graph of the function.
This leads to the following theorem.
(Note that, for this theorem, we do not make the distinction between the Euclidean and non-Euclidean case.)

\begin{theorem}\label{thm-holder-real}
Let $f$ be a differentiable function from $E$ (with norm $\mynor$) to $\RRb$.
Let $0\leq\nu\leq1$, and suppose there exist $L^-\geq0$ and $L^+\geq0$, such that, for every $x,y\in E$,
\begin{equation}\label{eq-error-diff}
\frac{-L^-}{1+\nu} \lVert x-y \rVert^{1+\nu} \leq f(y) - f(x) - \langle f'(x),y-x\rangle \leq \frac{L^+}{1+\nu} \lVert y-x \rVert^{1+\nu} .
\end{equation}
Then, $f$ has $\nu$-\HOL{} continuous gradient, and the $\nu$-\HOL{} constant $\MFNU$ satisfies%
\footnote{With the convention that $0^0=1$, as in Theorem~\ref{thm-holder-general}.}
\[
\MFNU \leq 2^{-\nu} \left(\frac{1+\nu}{\nu}\right)^\nu \left(L^-+L^+\right) \!.
\]\vskip0pt
\end{theorem}

\begin{proof}
The proof is similar to the proof of Theorem~\ref{thm-holder-general}.
Let us fix arbitrary $\xxb,\yyb\in E$.
Without loss of generality, we may assume that $\xxb = 0$, $f(\xxb) = 0$, and $f'(\xxb) = 0$.
Denote $\varphi = f'(\yyb)$.
Then, for any two directions $z_1,z_2\in E$, we have (from \eqref{eq-error-diff} with $(x,y)=(0,z_1)$ and with $(x,y)=(\yyb,z_1)$)
\[
\frac{-L^-}{1+\nu} \lVert z_1 \rVert^{1+\nu} \leq f(z_1) \leq f(\yyb) + \langle \varphi,z_1-\yyb\rangle + \frac{L^+}{1+\nu} \lVert z_1-\yyb \rVert^{1+\nu},
\]
and (from \eqref{eq-error-diff} with $(x,y)=(0,z_2)$ and with $(x,y)=(\yyb,z_2)$)
\[
\frac{L^+}{1+\nu} \lVert z_2 \rVert^{1+\nu} \geq f(z_2) \geq f(\yyb) + \langle \varphi,z_2-\yyb\rangle - \frac{L^-}{1+\nu} \lVert z_2-\yyb \rVert^{1+\nu} .
\]
Subtracting the rightmost and leftmost sides of the first set of inequalities from the second one, we get
\[
\langle \varphi,z_2-z_1\rangle \leq \frac{L^-}{1+\nu} \left( \lVert z_1 \rVert^{1+\nu} + \lVert z_2-\yyb \rVert^{1+\nu} \right) + \frac{L^+}{1+\nu} \left( \lVert z_2 \rVert^{1+\nu} + \lVert z_1-\yyb \rVert^{1+\nu} \right) \!.
\]
Let $\vvh\in E$ be such that $\lVert \vvh \rVert = 1$ and $\langle \varphi,\vvh\rangle=\lVert \varphi \rVert_*$.
Then, let $\alpha\geq0$, and define $z_1 = \frac{\yyb-\alpha\vvh}{2}$ and $z_2 = \frac{\yyb+\alpha\vvh}{2}$.
This gives
\begin{equation}\label{eq-bound-2}
\alpha \, \langle \varphi,\vvh\rangle \leq \frac{2}{1+\nu} \left( L^-\left\lVert \frac{\yyb-\alpha\vvh}{2} \right\rVert^{1+\nu} + L^+\left\lVert \frac{\yyb+\alpha\vvh}{2} \right\rVert^{1+\nu} \right) \!.
\end{equation}
Let $r=\lVert \yyb\rVert$.
The right-hand side of \eqref{eq-bound-2} can be bounded as follows:
\[
L^-\left\lVert \frac{\yyb-\alpha\vvh}{2} \right\rVert^{1+\nu} + L^+\left\lVert \frac{\yyb+\alpha\vvh}{2} \right\rVert^{1+\nu} \leq (L^-+L^+) \left(\frac{r+\alpha}{2}\right)^{1+\nu} .
\]
We conclude in the same way as for the proof of Theorem~\ref{thm-holder-general}.~\qed
\end{proof}

\begin{corollary}\label{cor-convex-case}
Let $f$ be $\nu$-approximable.
If $f$ is convex, then the upper bound in \eqref{eq-coeff-1} can be improved by a factor $1/2$: $\LFNU \leq \MFNU \leq 2^{-\nu} \left(\frac{1+\nu}{\nu}\right)^\nu \LFNU$.
\end{corollary}

%%%%%%%%%%%%%%%%%%%%%%%%%%%%%%%%%%%%%%%%%%%%%%%%%%%%%%%%%%%%%%%%%%%%%%%%%%%%%%%%%%%%%%%%%%%%%%%%%%%%
\section{Application to Functions with Lipschitz Continuous Gradient}\label{sec-corollaries}

In this section, we particularize Theorem~\ref{thm-holder-general} and Corollary~\ref{cor-convex-case} to the classical case of $\nu=1$ (Lipschitz continuity).
Moreover, we obtain that the bounds are tight, meaning that there exist functions, such that $\MFNU$ attains either the lower or upper bound in \eqref{eq-coeff-1} when $\nu=1$.
Recall that $f$ is $1$-approximable, if there exists an $L\geq0$, such that, for every $x,y\in E$,
\begin{equation}\label{eq-error-lip}
\lvert f(y) - f(x) - \langle f'(x),y-x\rangle \rvert \leq \frac{L}{2} \lVert y-x \rVert^2 .
\end{equation}
To simplify the notation, we will denote $L_f = L_f(1)$ and $M_f=M_f(1)$.

\begin{corollary}\label{cor-lip-bound}
Let $f$ be a differentiable function from $E$ (with norm $\mynor$) to $\RRb$.
Then, $f$ is $1$-approximable, if and only if the gradient of $f$ is Lipschitz continuous.
Moreover, we have the following bounds, depending on the assumptions on $f$ and $\mynor$:
\begin{enumerate}
    \item In general, we have $\frac{1}{2}M_f\leq L_f\leq M_f$.
    \item If $\mynor$ is Euclidean, then $L_f=M_f$.
    \item If $f$ is convex, then $L_f=M_f$.
\end{enumerate}\vskip0pt
\end{corollary}

The proof of the corollary directly follows from Theorem~\ref{thm-holder-general} for the first and second items.
The third item can be easily proved with Corollary~\ref{cor-convex-case}.
(Let us also mention that the case of convex functions is also proved, with different arguments, in \cite[Theorem~2.1.5]{nesterov2013introductory}, for example.)

We would like to know whether the bounds on $L_f$ that we obtained in the first item of Corollary~\ref{cor-lip-bound} are tight.
The second and third items directly give that the upper bound $L_f=M_f$ is reached, when $\mynor$ is Euclidean, or when $f(\cdot)$ is convex.
Hence, the main question is: ``Is there a normed space $(E,\mynor)$ and a function $f$, defined on $E$, such that $L_f = \frac{1}{2}M_f$?''
The answer is yes, as shown in the following example:

\begin{example}\label{exa-infinite}
Let $E=\RRb^2$ with the \emph{infinity} norm, $\lVert x\rVert = \max\,\{\,\allowbreak \lvert x^{(1)}\rvert,\lvert x^{(2)}\rvert \,\}$ for all $x=(x^{(1)},x^{(2)})^\top\in\RRb^2$.
(We use superscripts to denote the components of vectors in $\RRb^2$.)
Define $f(x)=(x^{(1)})^2-(x^{(2)})^2$.
The gradient of $f$ at $x$ satisfies $\langle f'(x),h\rangle = 2x^{(1)}h^{(1)}-2x^{(2)}h^{(2)}$ for all $h=(h^{(1)},h^{(2)})^\top\in E$.

We check that \eqref{eq-error-lip} holds with $L=2$.
Indeed, for any $x,y\in\RRb^2$,
\begin{align*}
\text{Left-hand side of \eqref{eq-error-lip}} &= \Big\lvert \big[y^{(1)}-x^{(1)}\big]^2 - \big[y^{(2)}-x^{(2)}\big]^2 \Big\rvert \\
&\leq \max\kern1pt \left\{\, \big\lvert y^{(1)}-x^{(1)} \big\rvert^2,\big\lvert y^{(2)}-x^{(2)} \big\rvert^2 \,\right\} \\
&= \lVert y-x \rVert^2 .
\end{align*}
On the other hand, if $x=(1,1)^\top$, and $y=(0,0)^\top$, then $\langle f'(x) - f'(y),h\rangle = 2h^{(1)}-2h^{(2)}$ for every $h\in\RRb^2$.
Taking $h=(1,-1)^\top$, we get
\[
\lVert f'(x) - f'(y) \rVert_* \geq \lvert \langle f'(x),h\rangle - \langle f'(y),h\rangle \rvert = 4 .
\]
Since $\lVert h\rVert=1$, it follows that $M_f\geq4$.
Hence, we have $L_f\leq 2 \leq \frac{1}{2} M_f$, which concludes the proof that $L_f=\frac{1}{2}M_f$ for some functions $f$.
\end{example}

Before ending this section, we would like to point out that there is still a gap between the first and the second items of Corollary~\ref{cor-lip-bound}.
Indeed, in the second item, we state: if the norm on $E$ is Euclidean, then $L_f=M_f$.
On the other hand, if $\mynor$ is any norm, then we can ensure that $\frac{1}{2}M_f\leq L_f$.
However, we do not say anything about the possibility of finding a space $E$ with a non-Euclidean norm, but such that every function on $E$ satisfies $L_f=M_f$.
We will prove in Section~\ref{sec-app-bilin} that \emph{this situation is impossible}.
In other words, if all functions on a space $E$ (with norm $\mynor$) satisfy $L_f=M_f$, then $\mynor$ is Euclidean.
In fact, it suffices to have $L_f=M_f$ for every quadratic function with rank $2$ (to be defined below), to conclude that the norm $\mynor$ is Euclidean.

%%%%%%%%%%%%%%%%%%%%%%%%%%%%%%%%%%%%%%%%%%%%%%%%%%%%%%%%%%%%%%%%%%%%%%%%%%%%%%%%%%%%%%%%%%%%%%%%%%%%
\section{Application: Norm of Quadratic Functions}\label{sec-app-bilin}

In this section, we apply the results of the previous section to quadratic functions, that is, functions defined by self-adjoint operators.
We will use Corollary~\ref{cor-lip-bound} to derive bounds on the norms of self-adjoint operators; see Proposition~\ref{pro-bilin-bounds-various}.
We do not believe that the results presented in Proposition~\ref{pro-bilin-bounds-various} are new.
(For instance, the case of $\mynor$ Euclidean is a standard result in functional analysis on Hilbert spaces; see, e.g., \cite{friedman1982foundations}.)
However, we will use these results---and more precisely the fact that they are a consequence of Corollary~\ref{cor-lip-bound}---to show that it is impossible to have a space $E$ with a non-Euclidean norm, but such that every function on $E$ satisfies $L_f=M_f$ (cf.~the last paragraph of the previous section).

\begin{definition}[Self-adjoint operator]\label{def-bilin}
Let $E$ be a real finite-dimensional vector space.
\begin{enumerate}
    \item A \emph{self-adjoint operator} on $E$ is a linear map $B:E\to E^*$, $x\mapsto Bx$, satisfying $\langle Bx,y\rangle = \langle By,x\rangle$ for every $x,y\in E$.
    \item If $B$ is a self-adjoint operator on $E$, we let $Q_B:E\to\RRb$ be the \emph{quadratic form} associated with $B$, defined by $Q_B(x)=\langle Bx,x\rangle$ for every $x\in E$.
\end{enumerate}\vskip0pt
\end{definition}

The set of self-adjoint operators on $E$ is a real (finite-dimensional) vector space.
Given a norm $\mynor$ on $E$, we define the following norm on this space:
\[
\lVert B\rVert = \max \, \{\, \langle Bx,y\rangle : \lVert x\rVert = \lVert y\rVert = 1 \,\} = \max\, \{\, \lVert Bx\rVert_* : \lVert x\rVert = 1 \,\} .
\]
We also define the norm of a quadratic form $Q_B$:
\[
\lVert Q_B\rVert = \max\, \{\, \lvert Q_B(x)\rvert : \lVert x\rVert = 1 \,\} = \max\, \{\, \lvert\langle Bx,x\rangle\rvert : \lVert x\rVert = 1 \,\} .
\]

The quantities $\lVert B\rVert$ and $\lVert Q_B^{}\rVert$ appear in several fields of mathematics:

\begin{example}\label{exa-self-adjoint}
If $E=\RRb^n$, then $B$ can be identified with a symmetric matrix $\BBt\in\RRb^{n\times n}$.
If $\lVert x\rVert=\sqrt{x^\top x}$ (canonical Euclidean norm), then $\lVert B\rVert=\lVert Q_B^{}\rVert=\rho(\BBt)$, where $\rho(\BBt)$ is the largest absolute value of the eigenvalues of $\BBt$.

If $\mynor$ is the $\ell_1$-norm, i.e., $\lVert x\rVert = \sum_{i=1}^n \lvert x^{(i)}\rvert$, where $x=(x^{(1)},\ldots,x^{(n)})^\top$, then $\lVert B\rVert=\max_{1\leq i,j\leq n}\, \lvert \BBt^{(i,j)}\rvert$, where $\BBt^{(i,j)}$ is the $(i,j)$th entry of matrix $\BBt$.
On the other hand, the value of $\lVert Q_B^{}\rVert$ is, in general, NP-hard to compute.
This quantity appears, for instance, in the problem of determining the largest clique in a graph (Motzkin--Straus theorem \cite{motzkin1965maxima}).

The quantity $\lVert Q_B^{}\rVert$, with $\lVert x\rVert=\max_{1\leq i\leq n}\, \lvert x^{(i)}\rvert$ (\emph{infinity} norm), appears, for instance, in the MAXCUT problem \cite[Section~4.3.3]{ben2001lectures}.
If $\BBt$ is a positive semidefinite matrix (as it is the case in the MAXCUT problem), then we can show that $\lVert Q_B^{}\rVert=\lVert B\rVert$ (see Proposition~\ref{pro-bilin-bounds-various}).
\end{example}

Clearly, $\lVert Q_B\rVert\leq \lVert B\rVert$.
The following proposition, which follows from Corollary~\ref{cor-lip-bound}, provides a partial converse result.

\begin{proposition}\label{pro-bilin-bounds-various}
Let $E$ be a normed vector space with norm $\mynor$, and let $B$ be a self-adjoint operator on $E$.
Then, we have the following results, depending on the assumptions on $E$ and on $B$:
\begin{enumerate}
    \item In general, we have $\frac{1}{2} \lVert B\rVert\leq \lVert Q_B\rVert\leq \lVert B\rVert$.
    \item If $\mynor$ is Euclidean, then $\lVert Q_B\rVert=\lVert B\rVert$.
    \item If $B$ is positive semidefinite%
    \footnote{That is, $\langle Bx,x\rangle\geq0$ for every $x\in E$.},
    then $\lVert Q_B\rVert=\lVert B\rVert$.
\end{enumerate}\vskip0pt
\end{proposition}

\begin{proof}
Let $f(x) = \frac{1}{2} Q_B(x) = \frac{1}{2} \langle Bx,x\rangle$.
The gradient of $f$ at $x$ is $f'(x)=Bx\in E^*$.
Let $x,y$ be in $E$.
The difference between $f(y)$ and its first-order approximation $f(x) + \langle f'(x),y-x\rangle$ is given by
\begin{align*}
\frac{1}{2}\langle By,y\rangle - \frac{1}{2}\langle Bx,x\rangle - \langle Bx,y-x\rangle &= \frac{1}{2}\langle By,y\rangle + \frac{1}{2}\langle Bx,x\rangle - \langle Bx,y\rangle \\
&= \frac{1}{2}\langle B(y-x),y-x\rangle \\
&= \frac{1}{2} Q_B(y-x) .
\end{align*}

Hence, $f$ is $1$-approximable, and $L_f=\lVert Q_B\rVert$.
It is also not hard to see that $f$ has Lipschitz continuous gradient, and $M_f=\lVert B\rVert$.
Thus, we obtain the first and second items of the proposition, from the first and second items of Corollary~\ref{cor-lip-bound}.
Finally, $f$ is convex, if and only if $B$ is positive semidefinite.
Hence, we get the desired results from the third item of Corollary~\ref{cor-lip-bound}.~\qed
\end{proof}

The second item in Proposition~\ref{pro-bilin-bounds-various} states that, when the norm on $E$ is Euclidean, the two quantities $\lVert B\rVert$ and $\lVert Q_B\rVert$ coincide for every $B$.
The following theorem shows that $\lVert Q_B\rVert=\lVert B\rVert$ for every self-adjoint operator $B$ on a normed vector space $E$, only if the norm of $E$ is Euclidean.%
\footnote{The result presented in Theorem~\ref{thm-eucl-bilin} seems to be a novel (to the best of the authors' knowledge) characterization of Euclidean norms in the finite-dimensional case.
(For a detailed survey of results on equivalent characterizations of Euclidean norms, we refer the reader to the celebrated book by Amir \cite{amir1986characterizations}.)}
(The proof of the theorem is presented in the ``Appendix.'')

\begin{theorem}\label{thm-eucl-bilin}
Let $(E,\mynor)$ be a finite-dimensional normed vector space.
Then, $\mynor$ is Euclidean, if and only if $\lVert Q_B\rVert=\lVert B\rVert$ for every self-adjoint operator $B$ on $E$.
\end{theorem}

\begin{remark}
From the proof of Theorem~\ref{thm-eucl-bilin} (see the ``Appendix''), we can obtain a stronger version of this theorem: $\mynor$ is Euclidean, if and only if $\lVert Q_B\rVert=\lVert B\rVert$ for every self-adjoint operator $B$ on $E$ \emph{with rank $2$}, that is, for every $B$ that can be expressed as $Bx = \langle\varphi,x\rangle\varphi \pm \langle\psi,x\rangle\psi$ for some $\varphi,\psi\in E^*$.
\end{remark}

Finally, Theorem~\ref{thm-eucl-bilin} allows us to answer the question in the last paragraph of the previous section:

\begin{corollary}
Let $(E,\mynor)$ be a finite-dimensional normed vector space, and suppose that, for every real-valued function $f$ on $E$, we have $L_f=M_f$ (using the notation of Corollary~\ref{cor-lip-bound}).
Then, $\mynor$ is Euclidean.
\end{corollary}

\begin{proof}
If $L_f=M_f$ for every function $f$ on $E$, this is also true for every quadratic function $f(x)=Q_B^{}(x)=\langle Bx,x\rangle$, where $B$ is any self-adjoint operator on $E$.
Hence, using similar developments as in the proof of Proposition~\ref{pro-bilin-bounds-various}, we get that $\lVert Q_B^{}\rVert=\lVert B\rVert$ for every self-adjoint operator $B$ on $E$.
Thus, Theorem~\ref{thm-eucl-bilin} implies that $\mynor$ must be Euclidean.~\qed
\end{proof}

%%%%%%%%%%%%%%%%%%%%%%%%%%%%%%%%%%%%%%%%%%%%%%%%%%%%%%%%%%%%%%%%%%%%%%%%%%%%%%%%%%%%%%%%%%%%%%%%%%%%
\section{Conclusions}

We have shown that \HOL{} continuity of the gradient is a sufficient and \emph{necessary} condition for a function to have a global upper bound on the error of its first-order Taylor approximation.
We established the link between the parameter appearing in the global upper bound on the error of the first-order Taylor approximation and the \HOL{} constant of the gradient: This takes the form of an interval, depending on the \HOL{} constant, in which the approximation parameter is guaranteed to be.

For the Lipschitz case, an example shows that the interval cannot be shortened.
On top of this, tighter bounds can be obtained, if we further assume that the function is convex, or if the underlying norm is Euclidean.
In particular, for the Lipschitz case, if the norm is Euclidean, then the interval reduces to a single point, and Theorem~\ref{thm-eucl-bilin} shows that the assumption that the norm is Euclidean is \emph{not} conservative.
We have not addressed in this paper the question of whether the interval for the general case (i.e., with \HOL{} exponent different than one) is tight or not.
We leave it for further work.

Another follow-up topic is a generalization of Theorem~\ref{thm-holder-general} to higher-order derivatives and higher-order Taylor approximations.
Those play central roles in higher-order optimization methods.
By applying a similar argument as in the proof of Theorem~\ref{thm-holder-general}, one can obtain an interval, depending on the \HOL{} constant of the derivative of a given order, in which the approximation parameter of the Taylor approximation (of order equal to the one of the derivatives) is guaranteed to be.
Drawing upon this observation, the following aspects still need to be addressed:
Are the obtained bounds tight?
Can we obtain better bounds, if we further assume that the norm is Euclidean?
We plan to study these questions in further work.

\begin{acknowledgements}
\singlespacing
This work was supported by (i)~the Fonds de la Recherche Scientifique -- FNRS and the Fonds Wetenschappelijk Onderzoek -- Vlaanderen under EOS Project no 30468160, (ii)~``Communaut\'e fran\c{c}aise de Belgique -- Actions de Recherche Concert\'ees'' (contract ARC 14/19-060).
The research of the first author was supported by a FNRS/FRIA grant.
The research of the third author was supported by the FNRS, the Walloon Region and the Innoviris Foundation.
The research of the fourth author was supported by ERC Advanced Grant 788368.
\end{acknowledgements}

%%%%%%%%%%%%%%%%%%%%%%%%%%%%%%%%%%%%%%%%%%%%%%%%%%%%%%%%%%%%%%%%%%%%%%%%%%%%%%%%%%%%%%%%%%%%%%%%%%%%
\appendix  %This command ends the counting of sections.
\section*{Appendix: Proof of Theorem~\ref{thm-eucl-bilin}}
\stepcounter{section}

The proof relies on the fact that, if $\mynor$ is not Euclidean, then the \emph{unit ball} defined by $\mynor$, i.e., $\{x\in E:\lVert x\rVert\allowbreak\leq1\}$, is not equal to the ellipsoid with smallest volume containing this ball.
Based on this ellipsoid, we will build a self-adjoint operator $B:E\to E^*$, such that $\lVert Q_B\rVert<\lVert B\rVert$.
The notions of ellipsoid and (Lebesgue) volume are defined on $\RRb^n$ only.
The following lemma implies, among other things, that there is no loss of generality in restricting to the case $E=\RRb^n$:

\begin{lemma}\label{lem-2dim-bijec-corr}
Let $E$ be a real vector space with norm $\mynor$, and let $A:E \rightarrow E'$ be a bijective linear map.
Then, $\mynor$ is Euclidean, if and only if the norm $\mynor'$ on $E'$, defined by $\lVert x\rVert' = \lVert A^{-1}x\rVert$, is Euclidean.
\end{lemma}

\begin{proof}
Straightforward from the definition of $\mynor$ being Euclidean, if and only if it is induced by a scalar product, i.e., if and only if there exists a self-adjoint operator $H:E\to E^*$, satisfying $\lVert x\rVert^2=\langle Hx,x\rangle$ for all $x\in E$.~\qed
\end{proof}

\begingroup%
\def\proofname{Proof of Theorem~\ref{thm-eucl-bilin}}%
\begin{proof}
The ``only if'' part follows from Proposition~\ref{pro-bilin-bounds-various}.
For the proof of the ``if'' part, let $E$ be an $n$-dimensional vector space, and let $\mynor$ be a non-Euclidean norm on $E$.
We will build a self-adjoint operator $B$ on $E$, such that $\lVert Q_B\rVert<\lVert B\rVert$.

By Lemma~\ref{lem-2dim-bijec-corr}, we may assume that $E=\RRb^n$ and that $\mynor$ is a non-Euclidean norm on $\RRb^n$.
We use superscripts to denote the components of vectors in $\RRb^n$: $x=(x^{(1)},\ldots,x^{(n)})^\top$.

Let $K=\{ x\in\RRb^n : \lVert x\rVert \leq 1 \}$.
Because $K$ is compact, convex, with non-empty interior, and symmetric with respect to the origin, the L{\"o}wner--John ellipsoid theorem \cite{john1948extremum,ball1992ellipsoids} asserts that there exists a unique ellipsoid $\EEc$, with minimal volume, and such that $K\subseteq\EEc$.
Moreover, $\EEc$ is centered at the origin, and $K$ has $n$ linearly independent vectors on the boundary of $\EEc$.

Let $L:\RRb^n\to\RRb^n$ be a linear isomorphism, such that $L\EEc$ is the Euclidean ball $\BBb^n=\{x\in\RRb^n : \lVert x\rVert_2\leq1\}$, where $\lVert x\rVert_2=\sqrt{x^\top x}$ is the canonical Euclidean norm on $\RRb^n$.
Let $\lVert x\rVert'= \lVert L^{-1}x\rVert$, and let $K'=\{ x\in\RRb^n : \lVert x\rVert' \leq 1 \}$.
By Lemma~\ref{lem-2dim-bijec-corr}, $\mynor'$ is not Euclidean.
Since $K'=LK$, it is clear that $K'$ is compact, convex, with non-empty interior, and symmetric with respect to the origin.
Moreover, $K'$ is included in $\BBb^n$, and it has $n$ linearly independent vectors on the boundary $\SSb^{n-1}$ of $\BBb^n$.

We will need the following lemma to conclude the proof of Theorem~\ref{thm-eucl-bilin}:

\begin{lemma}\label{lem-two-dim-reduction}
There exist $u,v\in\SSb^{n-1}\cap K'$, not colinear, such that $\frac{u+v}{\lVert u+v\rVert_2}\notin K'$.
\end{lemma}

We proceed with the proof of Theorem~\ref{thm-eucl-bilin} (a proof of Lemma~\ref{lem-two-dim-reduction} is provided at the end of this appendix).
Let $u,v$ be as in Lemma~\ref{lem-two-dim-reduction}, and define $e_1 = \frac{u+v}{\lVert u+v\rVert_2}$ and $e_2=\frac{u-v}{\lVert u-v\rVert_2}$.
Note that these vectors are orthonormal (w.r.t.~the inner product $x^\top y$).

Let $\kappa= \max\, \{ \lvert e_1^\top x\rvert : x\in K' \}$.
Since $\lvert e_1^\top x\rvert<1$ for every $x\in\BBb^n\setminus\{\pm e_1\}$, and $\pm e_1\notin K'$, we have that $\kappa<1$.
Moreover, $\kappa>0$, since $\mathrm{int}(K')\neq\varnothing$.
Let $\BBt$ be the self-adjoint operator on $\RRb^n$, defined by
\[
\langle \BBt x,y\rangle = \frac{1}{\kappa^2}\, (e_1^\top x)(e_1^\top y) - (e_2^\top x)(e_2^\top y)
\]
for every $x,y\in\RRb^n$.
Let $x\in K'$.
Then,
\[
-1\leq - (e_2^\top x)^2 \leq \langle \BBt x,x\rangle \leq \frac{1}{\kappa^2}\,(e_1^\top x)^2 \leq \frac{1}{\kappa^2}\,\kappa^2 = 1 .
\]
It follows that, for every $x\in\RRb^n$ with $x\neq0$, $\lvert \langle \BBt x,x\rangle\rvert= \lVert x\rVert'^2\,\lvert \langle \BBt \frac{x}{\lVert x\rVert'},\frac{x}{\lVert x\rVert'}\rangle\rvert\leq\lVert x\rVert'^2$.
Hence, $\lVert Q_{\BBt}\rVert\leq1$.
Now, we will show that $\lvert \langle \BBt u,v\rangle\rvert>\lVert u\rVert'\lVert v\rVert'$ (where $u,v$ are as above).
Therefore, let $\alpha=\lVert u+v\rVert_2$ and $\beta=\lVert u-v\rVert_2$.
Observe that $u=\frac{\alpha e_1+\beta e_2}{2}$ and $v=\frac{\alpha e_1-\beta e_2}{2}$.
Thus,
\[
\langle \BBt u,v\rangle = \frac{1}{\kappa^2}\,\frac{\alpha^2}{4} + \frac{\beta^2}{4} = \frac{1-\kappa^2}{\kappa^2}\,\frac{\alpha^2}{4} + \frac{\alpha^2}{4} + \frac{\beta^2}{4} .
\]
This shows that $\langle \BBt u,v\rangle>1$, since (by the parallelogram identity)
\[
\frac{\alpha^2}{4} + \frac{\beta^2}{4} = \frac{1}{4}\left( \lVert u+v\rVert_2^2 + \lVert u-v\rVert_2^2 \right) = 1 ,
\]
$0<\kappa<1$, and $\alpha>0$.
Since $u,v\in K'$ (i.e., $\lVert u\rVert',\lVert v\rVert'\leq1$), we have that $\lVert u\rVert'\lVert v\rVert'\leq1<\lvert \langle \BBt u,v\rangle\rvert$.
Thus, $\lVert\BBt\rVert>1$.

Finally, define the self-adjoint operator $B$ on $E$ by $\langle Bx,y\rangle = \langle\BBt Lx,Ly\rangle$.
It is clear, from the definition of $\mynor'$, that $\lvert \langle Bx,x\rangle\rvert\leq\lVert x\rVert^2$ for every $x\in E$ and $\lvert \langle Bx,y\rangle\rvert>\lVert x\rVert\lVert y\rVert$ for $x=L^{-1}u$ and $y=L^{-1}v$ (where $u,v$ are as above).
Hence, one gets $\lVert Q_B\rVert\leq1<\lVert B\rVert$.
This concludes the proof of Theorem~\ref{thm-eucl-bilin}.~\qed
\end{proof}
\endgroup

It remains to prove Lemma~\ref{lem-two-dim-reduction}.
The following proposition, known as \emph{Fritz John necessary conditions for optimality} will be useful in the proof of Lemma~\ref{lem-two-dim-reduction}:

\begin{proposition}[Fritz John necessary conditions \cite{john1948extremum}]\label{pro-fritz}
Let $S$ be a compact metric space.
Let $F(x)$ be a real-valued function on $\RRb^n$, and let $G(x,y)$ be a real-valued function defined for all $(x,y)\in\RRb^n\times S$.
Assume that $F(x)$ and $G(x,y)$ are both differentiable with respect to $x$ and that $F(x)$, $G(x,y)$, $\frac{\partial F}{\partial x}(x)$, and $\frac{\partial G}{\partial x}(x,y)$ are continuous on $\RRb^n\times S$.
Let $R=\{x\in\RRb^n:G(x,y)\leq0,\,\forall y\in S\}$, and suppose that $R$ is non-empty.

Let $x^*\in R$ be such that $F(x^*)=\max_{x\in R} F(x)$.
Then, there is $m\in\{0,\ldots,n\}$, and points $y_1,\ldots,y_m\in S$, and nonnegative multipliers $\lambda_0,\lambda_1,\ldots,\lambda_m\geq0$, such that (i)~$G(x^*,y_i)=0$ for every $1\leq i\leq m$, (ii)~$\sum_{i=0}^m \lambda_i >0$, and (iii)
\[
\lambda_0\frac{\partial F}{\partial x}(x^*)=\sum_{i=1}^m \lambda_i \frac{\partial G}{\partial x}(x^*,y_i).
\]\vskip0pt
\end{proposition}

We refer the reader to \cite{john1948extremum} for a proof of Proposition~\ref{pro-fritz}.

\begingroup%
\def\proofname{Proof of Lemma~\ref{lem-two-dim-reduction}}%
\begin{proof}
Consider the following optimization problem:
\begin{equation}\label{eq-optimization}
\begin{tabular}{l@{\;\;\;\;}l}
maximize & $\displaystyle F(x)\coloneqq\lVert x\rVert_2^2$ \tabularnewline
subject to & $\displaystyle G(x,y)\coloneqq x^\top y-1\leq0$ \quad for every $y\in\SSb^{n-1}\cap K'$, \rule{-1pt}{5mm}\tabularnewline
\end{tabular}
\end{equation}
with variable $x\in\RRb^n$.

First, we show that \eqref{eq-optimization} is bounded.
Suppose the contrary, and, for every $k\geq1$, let $x_k$ be a feasible solution with $\lVert x_k\rVert_2\geq k$.
Let $\xh_k=x_k/\lVert x_k\rVert_2$.
Taking a subsequence if necessary, we may assume that $\xh_k$ converges to some $\xh_*$, with $\lVert\xh_*\rVert_2=1$.
Since $\xh_k^\top y\leq 1/\lVert x_k\rVert_2$ for every $y\in\SSb^{n-1}\cap K'$, we have that $\xh_*^\top y \leq0$ for every $y\in\SSb^{n-1}\cap K'$.
By symmetry of $\SSb^{n-1}\cap K'$, it follows that $\xh_*^\top y =0$ for every $x\in\SSb^{n-1}\cap K'$, a contradiction with the fact that $\SSb^{n-1}\cap K'$ contains $n$ linearly independent vectors.
Hence, the set of feasible solutions of \eqref{eq-optimization} is bounded, and closed (as the intersection of closed sets), so that \eqref{eq-optimization} has an optimal solution, say $\xopt$.

We will show that $\lVert \xopt\rVert_2>1$.
Therefore, we use the fact that $K'\neq\BBb^n$.%
\footnote{Indeed, a norm $\mynor$ is completely determined by its unit ball $K$, via the identity $\lVert x\rVert = \inf\,\{\alpha>0:x/\alpha\in K\}$.
Since $\mynor'$ is not Euclidean, its unit ball cannot be equal to $\BBb^n$ (since $\BBb^n$ is the unit ball of the canonical Euclidean norm).}
Fix some $z\in\SSb^{n-1}\setminus K'$, and let $\eta=\max\,\{ z^\top y:y\in K' \}$.
Since $z^\top y<1$ for every $y\in\BBb^n\setminus\{z\}$, and $z\notin K'$, we have that $\eta<1$.
Let $x=z/\eta$.
From the definition of $\eta$, it is clear that $x$ is a feasible solution of \eqref{eq-optimization}.
Moreover, $\lVert x\rVert_2=\eta^{-1}>1$, so that $\lVert \xopt\rVert_2\geq\lVert x\rVert_2>1$.

The gradient of $F$ at $\xopt$ is equal to $2\xopt$.
Then, Proposition~\ref{pro-fritz} asserts that there exist vectors $y_1,\ldots,y_m\in\SSb^{n-1}\cap K'$, and nonnegative multipliers $\lambda_0,\lambda_1,\ldots,\lambda_m\geq0$, such that $\xopt^\top y_i=1$ for every $1\leq i\leq m$, $\sum_{i=0}^m \lambda_i>0$, and $\lambda_0\xopt=\sum_{i=1}^m \lambda_iy_i$.
If $\lambda_0=0$, then $0=\sum_{i=1}^m \lambda_i\xopt^\top y_i = \sum_{i=1}^m \lambda_i >0$, a contradiction.
Hence, $\lambda_0>0$.
Suppose that $y_1,\ldots,y_m$ are colinear.
This implies that all $y_i$'s must be parallel to $\xopt$ (because $\lambda_0\xopt=\sum_{i=1}^m \lambda_iy_i$ and $\lambda_0\xopt\neq0$), and since they are in $\SSb^{n-1}$, we have that $y_i=\pm \xopt/\lVert \xopt\rVert_2$, so that $\xopt^\top y_i=\lVert \xopt\rVert_2>1$ or $-\xopt^\top y_i=\lVert \xopt\rVert_2>1$.
This gives a contradiction, since $-y_i$ and $y_i$ are both in $\SSb^{n-1}\cap K'$ (by the symmetry of $\SSb^{n-1}\cap K'$).
Thus, there exist at least two non-colinear vectors $u,v\in\SSb^{n-1}\cap K'$ satisfying $\xopt^\top u=1$ and $\xopt^\top v=1$.

Let $e_1 = \frac{u+v}{\lVert u+v\rVert_2}$.
Since $u$ and $v$ are not colinear, $\lVert u+v\rVert_2<2$, and thus, $\xopt^\top e_1=2/\lVert u+v\rVert_2>1$.
This shows that $e_1\notin\SSb^{n-1}\cap K'$.
By definition, $e_1$ is in $\SSb^{n-1}$, so that $e_1\notin K'$, concluding the proof of the lemma.~\qed
\end{proof}
\endgroup

%%%%%%%%%%%%%%%%%%%%%%%%%%%%%%%%%%%%%%%%%%%%%%%%%%%%%%%%%%%%%%%%%%%%%%%%%%%%%%%%%%%%%%%%%%%%%%%%%%%%
%References
% BibTeX users  please use  \bibliographystyle{spmpsci_unsrt}. The option spmpsci_unsrt prints the references in JOTA format  in the order they are cited.  
%Otherwise, please use the following:

\bibliographystyle{spmpsci_unsrt}
\bibliography{myrefs}

\end{document}